\documentclass[a4paper,leqno,11pt]{amsart}

\usepackage{amsfonts,amssymb,verbatim,amsmath,amsthm,latexsym,textcomp,amscd}
\usepackage{latexsym,amsfonts,amssymb,epsfig,verbatim}
\usepackage{amsmath,amsthm,amssymb,latexsym,graphics,textcomp}
\usepackage{graphicx}
\usepackage{color}
\usepackage{url}
\usepackage{enumerate}
\usepackage[mathscr]{euscript}
\usepackage{caption}

\usepackage{pgf,tikz}

\input xy
\xyoption{all}

\usepackage{hyperref}

\theoremstyle{plain}
\newtheorem{theorem}{Theorem}[section]
\newtheorem{thm}[theorem]{Theorem}
\newtheorem{lemma}[theorem]{Lemma}
\newtheorem{cor}[theorem]{Corollary}
\newtheorem{prop}[theorem]{Proposition}

\theoremstyle{definition}
\newtheorem{defn}[theorem]{Definition}
\newtheorem{rmk}[theorem]{Remark}
\newtheorem{exam}[theorem]{Example}

\newenvironment{myeq}[1][]
{\stepcounter{theorem}\begin{equation}\tag{\thetheorem}{#1}}
{\end{equation}}

\newcommand{\uZ}{\underline{\mathbb{Z}}}

\newcommand{\uZp}{\underline{\mathbb{Z}/p}}
\newcommand{\uZpp}{\underline{\mathbb{Z}/p^2}}

\newcommand{\uZZ}{\underline{\mathbb{Z}/2}}
\newcommand{\mM}{\underline{M}}
\newcommand{\mN}{\underline{N}}
\newcommand{\uM}{\underline{M}}

\newcommand{\mA}{\underline{A}}

\newcommand{\Z}{\mathbb Z}

\newcommand{\DD}{\mathcal{D}}

\newcommand{\R}{\mathbb R}
\newcommand{\C}{\mathbb C}

\newcommand{\res}{\mathit{res}}

\newcommand{\Map}{\mathit{Map}}

\newcommand{\uH}{\underline{H}}

\newcommand{\tH}{\tilde{H}}
\newcommand{\bs}{\bigstar}
\newcommand{\tHbs}{\tilde{H}^\bigstar_G}
\newcommand{\tHbsb}{\tilde{H}^{\bigstar+1}_G}

\author{Samik Basu, Surojit Ghosh}

\email{samik.basu2@gmail.com; samikbasu@isical.ac.in}
\address{Stat-Math Unit,
Indian Statistical Institute,
B. T. Road, Kolkata-700108, India.}

\email{surojitghosh89@gmail.com; gsurojit@campus.haifa.ac.il}
\address{Department of Mathematics,
University of Haifa,
Haifa-3498838, Israel.} 


\begin{document}

\title{Bredon cohomology of finite dimensional $C_p$-spaces}
\maketitle

\begin{abstract}
For finite dimensional free $C_p$-spaces, the calculation of the Bredon cohomology ring as an algebra over the cohomology of $S^0$ is used to prove the non-existence of certain $C_p$-maps. These are related to Borsuk-Ulam type theorems, and equivariant maps related to the topological Tverberg conjecture. For certain finite dimensional $C_p$-spaces which are formed out of representations, it is proved that the cohomology is a free module over the cohomology of a point. All the calculations are done for the cohomology with constant coefficients $\Z/p$. 
\end{abstract}

\section{Introduction}
For the cyclic group $C_p$ of prime order $p$, computations of $RO(G)$-graded cohomology $\uH^\bigstar_G(S^0)$ by Stong and Lewis \cite{Lew88} allow us to compute the cohomology of $C_p$-spaces. In this paper, we prove some structural results about the cohomology of $C_p$-spaces with $\uZp$ coefficients. In many ways, $\uZp$ is the analogue of $\Z/p$-coefficients in ordinary cohomology in the non-equivariant case. 

For coefficients in a constant Mackey functor, the integer graded cohomology is the cohomology of the orbit space. If $X$ is a $C_p$-space with free action, we prove that the $RO(G)$-graded cohomology is determined from the cohomology of the orbit space. More precisely, 
$$H^\bigstar_{C_p}(X;\uZp) \cong H^\ast(X/C_p) \otimes_{\xi\in \hat{C_p}- \{1\}} \Z/p[u_\xi^\pm]$$ 
where $\hat{C_p}$ is set of characters of $C_p$ (see Proposition \ref{addstr}). 

We also compute the module structure of $H^\bigstar_{C_p}(X;\uZp)$ over the cohomology of a point (see Proposition \ref{modX}). The module structure allows us to rule out certain equivariant maps between free $C_p$-spaces. More precisely, we deduce Borsuk-Ulam type theorems \cite{IM00} : If $V$ and $V'$ are two fixed point free $C_p$-representations, there are no $C_p$-maps $S(V) \to S(V')$ if $\dim(V) > \dim(V')$ (see Corollary \ref{borsul}). The module structure calculations also allow us to deduce the topological Tverberg conjecture in the prime case, first proved by B\'ar\'any, Shlosman, and Sz\"{u}cs  \cite{B-S-S} (see Theorem \ref{obs}). 

If $X$ is a finite dimensional free $C_p$-space, we use the module structure to obtain a numerical bound $n(X)$ such that for every $i >n(X)$ the elements of $\uH^\bigstar_G(S^0)$ of degree $i$ operate trivially on $\uH^\bigstar_G(X_+)$. This number is related to the Fadell-Husseini index \cite{FH} (see Theorem \ref{F-H}). 

We also prove a freeness result for certain $C_p$-spaces. This kind of theorem writes the cohomology of $X$ as a free module over the cohomology of a point, for certain $G$-complexes formed out of cells in unitary $G$-representations. In ordinary cohomology, with $\Z$-coefficients this is true if $X$ has even dimensional cells, and with $\Z/p$ coefficients this is true for any $X$. For the group $C_p$, Lewis \cite{Lew88} proved a freeness result for certain even dimensional complexes.
There is a condition on the fixed points to deduce that the attaching maps induce the $0$ map on cohomology. For this theorem, only the additive structure of the $RO(C_p)$-graded cohomology of $S^0$ is used. 

Ferland and Lewis \cite{FL04} use the ring structure of the  $RO(C_p)$-graded cohomology of $S^0$ to remove the condition on the fixed points. They show that for any $C_p$-cell complex formed out of representations with even cells, the cohomology is a free module over the cohomology of $S^0$. In this case, the attaching maps do not induce the $0$ map on cohomology but it is proved that the cohomology of the cofibre is still free. This crucially uses the ring structure of $\uH^\bigstar_{C_p}(S^0)$. 

Kronholm \cite{Kr10} has proved that any $C_2$-space obtained inductively by attaching representation cells in increasing dimension (called $Rep(C_2)$-complexes),  the cohomology is a free module over the cohomology of a point when the coefficients are $\uZZ$. This also involves a careful analysis using the ring structure of $\uH^\bigstar_{C_2}(S^0; \uZZ)$.  In this paper, we make the analogous analysis for $C_p$ where $p>2$, and prove that for certain $Rep(C_p)$-complexes, the cohomology is free as a module over the cohomology of a point when the coefficients are $\uZp$. It is important to note here that the direct generalization does not work in the $p$ odd case, due to an example of Clover May (see Remark \ref{Clover-countex}). We prove the theorem in the case where the complex does not have cells in consecutive dimension (see Theorem \ref{frcp}).

\subsection{Organization} In Section 2, we recall some definitions and results from $RO(G)$-graded Bredon cohomology theory. We use these methods to write down the cohomology of free $C_p$-spaces in Section $3$. In Section $4$, we describe some applications such as the proof of the topological Tverberg conjecture in the prime case, and a freeness theorem for representation complexes.

\subsection{Acknowledgements}
The authors would like to thank Clover May for pointing out the conterexample (Remark \ref{Clover-countex}) regarding an error in an earlier version of the paper. The research of the first author was supported by SERB MATRICS grant number 2018/000845.

\section{Preliminaries}

We recall certain basic ideas and techniques in Bredon cohomology. Along the way we fix the notations used throughout the paper. The details and proofs of the stated facts may be found in \cite{May96}. The notation $G$ will be used for the cyclic group $C_p$ of prime order $p$, though most of the facts in this section also holds for a general finite group.  

For a orthogonal representation $V$ of the group $G,$ let $D(V)$ and $S(V)$ denote the unit disc and unit representation sphere in $V,$
 with the action induced from that of $V$. A $G$-CW complex is a $G$-space $X$ with a filtration $\{X^{(n)} \}_{n \geq 0}$, where $X^{(0)}$ is a disjoint union of $G$-orbits, and $X^{(n)}$ is obtained from $X^{(n-1)}$ by attaching cells of the form $G \times_H D^n$ along maps $G \times_H S^{n-1} \to X^{(n-1)}$ where $H \leq G$ and the action of $G$ on $D^n$ and $S^{n-1}$ are trivial. The space $X^{(n)}$ is defined as the $n^{th}$-skelelon of $X.$ The attaching map $G \times_H S^{n-1} \to X^{(n-1)}$ is equivalent to the map $S^{n-1} \to (X^{(n-1)})^H= (X^H)^{(n-1)}$. One may deduce that the category of $G$-CW complexes is equivalent to the functor category from $\mathcal{O}_G$ to the category of CW-complexes, where $\mathcal{O}_G$ is the orbit category of $G$ with objects are the finite $G$-sets and morphisms are the $G$-equivariant maps between $G$-sets.
 
 A coefficient system for the group $G$ is a contravariant functor from $\mathcal{O}_G$ to the category of abelian groups. Since this diagram category from $\mathcal{O}_G$ to the abelian group category is an additive category, therefore, we can talk about chains of coefficient systems. In particular, for a $G$-space we define the Bredon chain, $\underline{C}_*(X; \mathbb{Z}),$ given by assignment $G/H \to C_*(X^H; \mathbb{Z}).$
 
 \begin{defn}
 Let $\underline{M}$ be a coefficient system for a group $G$ and $X$ be a $G$-space. Define the $n^{th}$-Bredon cochains of $X$ with coefficients in $\underline{M}$ as $C^n_G(X; \underline{M}) = Hom_{\mathcal{O}_G}(\underline{C}_n(X, \mathbb{Z}); \underline{M}).$ The cohomology of this complex is defined as $\mathbb{Z}$-graded Bredon cohomology of $X$ with coefficients in $\underline{M}$ and denoted by $H^*_{G}(X; \underline{M}).$
 \end{defn}
 
 Equivariant homotopy and cohomology theories are more naturally graded on $RO(G),$ the Grothendieck group of finite real orthogonal representations of $G.$ To obtain this kind of theory one needs more structure on the coefficient systems. These are called Mackey functors.

\begin{defn}
A Mackey functor consists of a pair $\uM = (\uM_\ast , \uM^\ast )$ of functors from the category of finite $G$-sets to $\mathcal{A} b$, with $\uM_\ast$ covariant and $\uM^\ast$ contravariant. On every object $S$, $\uM^\ast$ and $\uM_\ast$ have the same value which we denote by $\uM(S)$,  and $\uM$ carries disjoint unions to direct sums. The functors are required to satisfy that for every pullback diagram of finite $G$-sets as below 
$$\xymatrix{ P \ar[r]^\delta    \ar[d]^\gamma                             & X \ar[d]^\alpha \\ 
                        Y \ar[r]^\beta                                                     &  Z,}$$
one has $\uM^\ast(\alpha) \circ \uM_\ast(\beta) = \uM_\ast(\delta) \circ \uM^\ast(\gamma).$ 
\end{defn}

Mackey functors are naturally contravariant functors from the Burnside category $Burn_G$ of $G$ to abelian groups. The objects of $Burn_G$ are finite $G$-sets and the morphisms are formed by group completing the monoid of correspondences. The representable functor associated to the $G$-set $G/G$ is called the Burnside ring Mackey functor $\mA$. For a finite $G$-set $S$, $\mA_S$ is the representable functor associated to $S$. 

\begin{exam}For an abelian group $C,$ an immediate example for a Mackey functor is the constant Mackey functor $\underline{C}$ defined by the assignment $\underline{C}(S) = \Map^G(S, C),$ the set of $G$-maps from the $G$-orbit $S$ to $C$ with trivial $G$-action.
\end{exam}

Equivariant cohomology theories are represented by $G$-spectra. The naive $G$-spectra are those in which only the desuspension with respect to trivial $G$-spheres are allowed. Usually, what we mean by $G$-spectra are those in which desuspension with respect to all representation spheres are allowed. In the viewpoint of \cite{LMS86}, naive $G$-spectra are indexed over a trivial $G$-universe and $G$-spectra are indexed over a complete $G$-universe. As we are allowed to take desuspension with respect to representation-spheres, the associated cohomology theories become $RO(G)$-graded. 

We consider the orthogonal $G$-spectra with positive complete model structure to model the equivariant stable homotopy theory, which can be read off from \cite[Appendix A, B]{HHR16}. In particular, $Sp^G$ denotes the symmetric model category of orthogonal $G$-spectra. We denote the homotopy class of maps by $[-, -]^G$ and the equivariant function spectrum $F(-,-)$, a right adjoint to the smash product $\wedge$ in $Sp^G$. 

Every $G$-set $S$ gives a suspension spectrum $\Sigma^\infty_G S_+$ in the category of $G$-spectra. It turns out that the category with finite $G$-sets as objects and homotopy classes of spectra maps as morphisms is naturally isomorphic to the Burnside category. Thus, the homotopy groups of $G$-spectra are naturally Mackey functors. For an equivariant orthogonal spectrum $X,$  we use $\pi_{\bigstar}(X)$ for its $RO(G)$-graded homotopy groups. In particular, for $\alpha = V -W \in RO(G),$ 
$$\pi_\alpha(X) = [S^V, S^W \wedge X]^G.$$

In non-equivariant homotopy theory, for each abelian group $A,$ there is a construction of the Eilenberg-Mac Lane spectra $HA$ satisfying  $$\pi_n(HA) = \begin{cases} A, & \text{if }n =0\\ 
0, & \text{otherwise.}
 \end{cases}$$
 
In the category of orthogonal spectra, one also has a construction of Eilenberg-Mac Lane spectra for each Mackey functor.

\begin{prop}
Let $\uM$ be a $G$-Mackey functor. Then there exist an equivariant Eilenberg-Mac Lane spectrum $H\uM$, unique up to homotopy in $Sp^G.$ 
\end{prop}
\begin{proof}
See \cite[Theorem 5.3]{GM95a}.
\end{proof} 

This implies a definition of $RO(G)$-graded homotopy groups of Equivariant Eilenberg-Mac Lane spectra which, unlike the integer graded ones, are non-trivial in infinitely many degrees. In fact, equivariant Eilenberg-Mac Lane spectra must arise from Mackey functors.  This is a theorem of Lewis, May, and McClure which we refer from Chapter XIII of \cite{May96}.  Therefore, we may argue that the integer-graded cohomology associated to coefficient systems extends to $RO(G)$-graded cohomology theories if and only if the coefficient system has an underlying Mackey functor structure.

\begin{defn}
 An $RO(G)$-graded cohomology theory consists of functors $E^\alpha$ for $\alpha\in RO(G)$, from reduced equivariant CW complexes to abelian groups which satisfy the usual axioms - homotopy invariance, excision, long exact sequence and the wedge axiom. 
\end{defn}
It is interesting to note that the suspension isomorphism for $RO(G)$-graded cohomology theories takes the form 
$E^{\alpha}(X) \cong E^{\alpha + V}(S^V \wedge X)$ for every based $G$-space $X$ and representation $V$.

We recall that there are change of groups functors on equivariant spectra. The restriction functor from $G$-spectra to $H$-spectra has a left adjoint given by smashing with $G/H_+$. This also induces an isomorphism on cohomology with Mackey functor coefficients 
$$\tilde{H}^\alpha_G(G/H_+\wedge X ; \uM)\cong \tilde{H}^\alpha_H(X; \res_H(\uM))$$

The $RO(G)$-graded theories may also be assumed to be Mackey functor-valued as in the definition below.   
\begin{defn}
Let $X$ be a pointed $G$-space, $\uM$ be any Mackey functor, $\alpha \in RO(G)$. Then the Mackey functor valued cohomology $\uH^{\alpha}_{G}(X;\uM)$ is defined as 
$$\uH^{\alpha}_{G}(X;\uM)(G/K) = \tilde{H}^{\alpha}_{G}({G/K}_+ \wedge X;\uM).$$
The restriction and transfer maps are induced by the appropriate maps of $G$-spectra. 
\end{defn}

The Mackey functor valued cohomology is always the reduced version, so we may only evaluate it on a pointed $G$-space. This means that we have $G$-invariant basepoint. On the other hand, the Bredon cohomology groups are defined for unpointed $G$-spaces also, so when we write the reduced version we put $\sim$ as a superscript. This explains our notation in the rest of the document where we use the notation $\tilde{H}^\bigstar_G(X_+)$. 

The next natural question is : {\it What is the structure on a Mackey functor which induces a ring structure on the cohomology of spaces?} There is a box product $\Box$ on the category of Mackey functors. For two Mackey functors $\mM, \mN$, this is obtained by taking the left Kan extension along 
$$\xymatrix{ Burn_G \times Burn_G \ar[r] \ar[d] & Ab \\ 
 Burn_G \ar@{-->}[ru]}$$
The right arrow in the top row is given by $(S,T)\mapsto \mM(S)\otimes \mN(T)$. The left vertical arrow is given by $(S,T) \mapsto S\times T$. The Mackey functors inducing ring valued cohomology theories are monoids under the box product $\Box$. The Burnside ring Mackey functor $\mA$, and the constant Mackey functors $\uZ$, $\uZp$ are monoids under $\Box$. Therefore $\tilde{H}^\bigstar_G(S^0;\mM)$ has a graded ring structure for $\mM= \mA, \uZ$ or $\uZp$. 

 \section{Cohomology of free $C_p$-spaces} \label{coh}
In this section, we compute the Bredon cohomology of free $C_p$-spaces. There are two ingredients in this, first the ring structure on the cohomology and second the module structure over the cohomology of a point. The former is computed in Section \ref{freess}, and the latter in Section \ref{module}. Along the way we also describe the method to compute $\tilde{H}^\bigstar_G(S^0)$ using the Tate square, to recalculate the computations by Stong and Lewis \cite{Lew88}. 

\subsection{The cohomology ring structure} \label{freess}
Let $X$ be a free $G$-space. We prove that the cohomology ring of $X$ is obtained in a neat way from the ordinary cohomology of $X/G$. For this purpose, we  set up a spectral sequence to compute the Bredon cohomology of a free $G$-space $X$ arising from a CW structure on $X$, which we then proceed to compute completely. This is the homotopy fixed point spectral sequence in the case $X=EG$ in \cite[Proposition 2.8]{HM17}. 

We start with a $G$-CW structure on $X$ as $X=\cup_s X^{(s)}$. Since $X$ is free the cells are of the type $G/e \times \DD^k$. We also note that $X/G$ has a CW complex structure with associated filtration $X^{(s)}/G$. Thus, we have 
$$X^{(s)}/ X^{(s-1)}\cong \bigvee_{e \in I(s)}G_+ \wedge S^s$$
For $\alpha \in RO(G)$, define 
$$E_1^{s,t}(\alpha) = \pi^G_{t-s}(F(X^{(s)}/ X^{(s-1)}, S^{-\alpha}\wedge H\uZp)).$$  
We make the following identifications
  \begin{align*}
  {\pi}^G_{t-s}(F(X^{(s)}/ X^{(s-1)}, S^{-\alpha}\wedge H\uZp)) & \cong {[S^{t-s}, F(\bigvee_{e \in I(s)}G_+ \wedge S^{s}, S^{-\alpha}\wedge H\uZp)]}^G \\
                                                                                       & \cong \bigoplus_{e \in I(s)}{[S^{t-s} \wedge G_+ \wedge S^s, S^{-\alpha}\wedge H \uZp]}^G \\ 
                                                                                                     & \cong \bigoplus_{e \in I(s)}  [S^t, S^{-\dim(\alpha)}\wedge H\Z/p] \\ 
                                                                                      & \cong \bigoplus_{e \in I(s)} {\pi}_t(S^{-\dim(\alpha)}\wedge H\Z/p) \\ 
                                                                                       & \cong C^s(X/G;{\pi}_t(S^{-\dim(\alpha)}\wedge H\Z/p )).
\end{align*}
In the case $p$ odd, all the representation spheres are orientable so that the action of $G$ on ${\pi}_t(S^{-\dim(\alpha)}\wedge H\Z/p)$ is trivial. In the case $p=2$, also the action of $G$ on ${\pi}_t(S^{-\dim(\alpha)}\wedge H\Z/2) \cong \Z/2$ is trivial. It follows that the boundary $d_1$ matches with the cellular coboundary. Therefore, we have proved the following result. 
\begin{prop} \label{HFPSS} There is a spectral sequence 
$$E^{s,t}_2(\alpha) = H^s(X/G; {\pi}_t(S^{-\dim(\alpha)}\wedge H\Z/p )) \Rightarrow \pi_{\alpha+t-s}^G(F(X_+, H\uZp)) $$
with boundary $d_r : E_r^{s,t}(\alpha) \to E_r^{s+r,t-r+1}(\alpha).$ The spectral sequences assemble together for various $\alpha$ into a multiplicative $RO(G)$-graded spectral sequence 
$$E^{s,\alpha}_2 = H^s(X/G; \pi_0 (S^{-\dim(\alpha)}\wedge H\Z/p)) \Rightarrow \pi_{\alpha -s}^G(F(X_+,H\uZp))$$ 
where $s\in \Z$ and $\alpha \in RO(G)$. 
\end{prop}
%
%
 
%
%

The multiplicative structure of the spectral sequence follows verbatim from \cite[Theorem 6.1]{Dug03} applied to $X$ instead of $EG$. We readily deduce that for every $\alpha$, the spectral sequence $E_2^{s,t}(\alpha)$ is concentrated at $t=-\dim(\alpha)$. Therefore, the spectral sequences degenerate at the second page. Also, the multiplicative structure on the spectral sequence in Proposition \ref{HFPSS} implies that everything is a product of the form $H^s(X/G) \otimes \pi_0 (S^{\dim(\alpha)} \wedge H\Z/p)$. Now $\pi_0 (S^{\dim(\alpha)} \wedge H\Z/p)$ fits together as a ring $\bigotimes\limits_{\xi \in \hat{G}} \Z[u_\xi^\pm]$ with $|u_\xi|=\dim(\xi)- \xi.$  Thus we obtain 
\begin{prop} \label{addstr} 
For a free $G$-space $X$, 
$$\tilde{H}^\bigstar_{G}(X_+;\uZp) \cong H^{\ast}(X/G;\Z/p)\otimes \bigotimes_{\xi \in \hat{G} \setminus \{1 \}} \Z/p [u_\xi^\pm].$$
\end{prop}
We observe here that the classes $u_\xi$ here are defined via the isomorphism $\pi_0(S^{\dim(\xi - 2)}\wedge H\Z/p)\cong \Z/p$ and hence, up to a unit in $\Z/p$. In the next section, we will make this choice unambiguous.  

\subsection{Cohomology of $S^0$} 
In this section, we compute the cohomology of $S^0$ with $\uZp$ coefficients. This was first computed by Stong and Lewis \cite{Lew88}. Here, we use the Tate square 
\begin{myeq}\label{tate}
\xymatrix{EG_+ \wedge H \uZp\ar[d]^{\simeq} \ar[r]& H \uZp \ar[r]\ar[d]^{q}    & \widetilde{EG}\wedge H \uZp \ar[d] \\ 
        EG_+ \wedge F(EG_+, H\uZp) \ar[r]              & F(EG_+, H\uZp) \ar[r]  & \widetilde{EG}\wedge F(EG_+, H\uZp)}.
\end{myeq} 

We first recall the definition of the classes $a_\xi$ and $u_\xi$ in $\tilde{H}^\bigstar_G(S^0; \uZp)$ from \cite{HHR16}.
\begin{defn} \label{cohgen}
 (1) For an actual representation $V$ with $V^G =0,$ let $a_V \in \pi_{-V}^G(S^0)$ be the maps $S^0 \to S^V$ which embeds $S^0$ to $S^V$ to $0$ and $\infty$ in $S^V.$ We will also use $a_V$ for the dual of its Hurewicz image in $\tilde{H}^V_G(S^0; \uZp).$

(2) For an actual orientable representation $W$ of dimension $n,$ let $u_W$ be a generator of $\underline{H}_n (S^W; \uZp)(G/G)$ which restricts to the choice of orientation in  $ \underline{H}_n (S^W; \uZp)(G/e)\cong \tilde{H}_n(S^n; \Z/p).$
\end{defn}

In homotopy grading, $u_W \in \underline{\pi}_{n -W}(H \uZp)(G/G).$ We also denote its dual by $u_W$ in $\tilde{H}^{W-n}_G(S^0; \uZp).$

The cases $p$ odd and $p=2$ are slightly different. We start with the $p$ odd case. Note that in this case there are $\frac{p-1}{2}$ non-trivial real irreducible $C_p$-representations which may be listed as $\{\xi^i| 1\leq i \leq \frac{p-1}{2}\}$ with $\xi$ the usual multiplication by the $p^{th}$ roots of $1$. From now on, we fix the notation $\xi$ for this representation. In the calculation of $\tilde{H}^\bigstar_G(S^0;\uZp) \cong \pi_{-\bigstar} H\uZp$, one notes that it suffices to restrict $\bigstar$ to the values $m+n\xi$ with $m,n \in \Z$. For, one may observe from \cite[Appendix B]{Fer99} that $S^\alpha \wedge H\uZp \simeq H\uZp$ whenever both the dimension of $\alpha$ and the dimension of $\alpha^G$ are $0$. It follows that $S^{\xi - \xi^i} \wedge H\uZp \simeq H\uZp$, so that there are invertible classes in the graded commutative ring $\pi_\bs H\uZp$ in degrees $\xi^i - \xi$, which makes 
\begin{myeq}\label{repfr}
\tilde{H}^{m+\sum_i n_i \xi^i}_G(S^0;\uZp) \cong \tilde{H}^{m+(\sum n_i)\xi}_G(S^0;\uZp).
\end{myeq}
From now onwards $\bigstar$ is restricted to degrees  $m+n\xi$ for $m,n \in \Z$, and $\ast $ is restricted to integer gradings. Now Definition \ref{cohgen} gives two classes $u_\xi \in  \tilde{H}^{\xi-2}_G(S^0; \uZp)$ and  $a_\xi \in \tilde{H}^\xi_{G}(S^0; \uZp)$. 

From Proposition \ref{addstr}, $\tilde{H}^\bigstar_G(EG_+;\uZp) \cong H^\ast(BG; \Z/p) \otimes \Z/p[u_\xi^\pm]$. We now check that under $q$ the class $u_\xi$ maps to a generator of $\tH^{\xi-2}_G(EG_+;\uZp)$. This is true for any free $G$-space whose cohomology is described in Proposition \ref{addstr}.  From the construction of the spectral sequence in Proposition \ref{HFPSS}, for $X = G/e$ we have $\tilde{H}^\bigstar_G(G/e_+; \uZp) \cong \Z/p[u_\xi^\pm].$ Note that the map 
$$\uH^{\xi-2}_G(S^0; \uZp)(\cong \uZp) \to \uH^{\xi-2}_G(G/e_+; \uZp)(\cong \mA_{G/e}\otimes \Z/p)$$
 induced by the  projection $G/e_+ \to S^0$ is an isomorphism at $G/G.$ Therefore, the class $u_\xi \in \tilde{H}^{\xi-2}_G(S^0; \uZp)$ maps to a generator of $u_\xi \in \tilde{H}^{\xi-2}_G(X_+; \uZp)$ for any $G$-space $X.$ This may be treated as the class $u_\xi$ in Proposition \ref{addstr}. 

We next calculate the image of $a_\xi \in \tilde{H}^\xi_G(S^0; \uZp)$ in $ \tilde{H}^\bigstar_{G}({EC_p}_+; \uZp)$, that is the image under $q$ in the diagram \eqref{tate}. We have
$$\tilde{H}^\bigstar_{G}({EC_p}_+; \uZp) \cong \Z/p[x, y, u_\xi^\pm]/(x^2) \text{ where } |x| =1, |y|=2, |u_\xi| = \xi-2.$$
Observe, $\tilde{H}^\xi_{G}({EC_p}_+; \uZp) \cong \Z/p\{u_\xi y \}.$ Also observe that the map 
$$\tilde{H}^\xi_{G}({EC_p}_+; \uZp) \to \tilde{H}^\xi_{G}(S(2\xi)_+; \uZp)$$ 
is an isomorphism (for $H \leq C_p$, the pair $({EC_p}_+^H, S(2\xi)_+^H)$ is $3$-connected so, $\tilde{H}^\xi_G({EC_p}_+, S(2\xi)_+; \uZp) =0$).
The cofibre sequence 
$$S(2\xi)_+ \to S^0 \to S^{2\xi}$$
induces the exact sequence 
$$\tilde{H}^{-\xi}_G(S^0; \uZp) \to \tilde{H}^{\xi}_G(S^0; \uZp)\to \tilde{H}^{\xi}_G(S(2\xi)_+; \uZp) \to \tilde{H}^{1-\xi}_G(S^0; \uZp)$$
One readily computes using the cofibre sequence
$$S(\xi)_+ \to S^0 \to S^{\xi}$$
that $\tilde{H}^{-\xi}_G(S^0; \uZp)=0$ and $\tilde{H}^{1-\xi}_G(S^0; \uZp)=0$. Therefore, we have 
$$\tilde{H}^{\xi}_G(S^0; \uZp) \cong \tilde{H}^\xi_{G}(S(2\xi)_+; \uZp).$$
 Thus, we may fix the generator $y$ such that $a_\xi \mapsto u_\xi y.$ Note that as $\beta(x) =y,$ this also fixes the generator $x$ (here $\beta$ refers to the Bockstein in ordinary cohomology). Therefore, we may write 
$$\pi_\bigstar F({EC_p}_+, H\uZp) \cong \Z/p[x, a_\xi, u_\xi^\pm]/(x^2).$$
Now we work around the Tate square \eqref{tate}. From \cite[Proposition 1.1]{GM95}, it follows that the left vertical arrow is a $G$-equivalence. The maps $S^0 \to \widetilde{EG}$ induce a localization with respect to $a_\xi,$ as $\widetilde{EG} \simeq \varinjlim\limits_{n} S^{n \xi}.$ This implies 
$$\pi_\bigstar \widetilde{EG}\wedge F({EG}_+, H\uZp) \cong \Z/p[x, a_\xi^\pm, u_\xi^\pm]/(x^2).$$ 
Now the map 
$$\Z/p[x, a_\xi, u_\xi^\pm]/(x^2) \to \Z/p[x, a_\xi^\pm, u_\xi^\pm]/(x^2)$$
 is injective, so, 
$$\pi_\bigstar {EG}_+\wedge F({EG}_+, H\uZp) \cong \bigoplus \Z/p \{\Sigma^{-1} \frac{x^\epsilon u_\xi^j}{a_\xi^k} \}\text{ for } \epsilon \in \{ 0,1\}, j \in \Z, k \geq 1.$$ 
It follows that 
$$\pi_\bigstar {EG}_+ \wedge H\uZp \cong \bigoplus \Z/p \{\Sigma^{-1} \frac{x^\epsilon u_\xi^j}{a_\xi^k} \}\text{ for } \epsilon \in \{ 0,1\}, j \in \Z, k \geq 1.$$ 
We know that $\pi_\bigstar \widetilde{EG}\wedge H \uZp$ is $a_\xi$-periodic, so it suffices to compute $\pi_n \widetilde{EG}\wedge H \uZp$ for $n \in \Z$ and then 
$$\pi_\bigstar \widetilde{EG}\wedge H \uZp \cong \pi_\ast (\widetilde{EG}\wedge H \uZp) [a_\xi^\pm].$$ 
Now, 
$$\pi_0({EG}_+ \wedge H \uZp) \to \pi_0 (H\uZp)$$ 
is the transfer map of $\uZp$ and hence $0.$ Therefore, we get 
$$\pi_\ast \widetilde{EG}\wedge H \uZp \cong \Z/p\{1 \} \oplus \Z/p \{ \frac{x^\epsilon u_\xi^k}{a_\xi^k}\}\text{ for } \epsilon \in \{0,1 \}, k \geq 1.$$ 
We define the class $\kappa_\xi$ as $u_\xi x$. Therefore, 
$$\pi_\bigstar \widetilde{EG}\wedge H \uZp \cong \Z/p[a_\xi^\pm , u_\xi, \kappa_\xi]/ (\kappa_\xi^2).$$
 We now compute the kernel of the boundary $\pi_\bigstar \widetilde{EG}\wedge H \uZp \to \pi_{\bigstar-1} {EG}_+\wedge H \uZp$ to be 
$$\Z/p[a_\xi, u_\xi, \kappa_\xi]/(\kappa_\xi^2)$$ 
and the cokernel to be 
$$ \Z/p \{\Sigma^{-1} u_\xi^{-j} a_\xi^{-k} \} \oplus  \Z/p \{\Sigma^{-1} \kappa_\xi u_\xi^{-j} a_\xi^{-k} \}\text{ for } j,k >0.$$
Therefore, we obtain 
\begin{prop} \label{cohptodd}
For $p$ odd,
$$\pi_\bigstar (H\uZp)\cong \Z/p[a_\xi, u_\xi, \kappa_\xi]/(\kappa_\xi^2)\oplus \Z/p \{\Sigma^{-1} \frac{1}{u_\xi^{j} a_\xi^{k}} \} \oplus \Z/p\{\Sigma^{-1}\frac{\kappa_\xi}{u_\xi^{j}a_\xi^k}\}~ j,k >0.$$
\end{prop}
This method was carried out in \cite[Proposition 6.3]{MZ17} for $H\uZ$. It is instructive to compare the results. We call the part $\Z/p[a_\xi, u_\xi, \kappa_\xi]/(\kappa_\xi^2)$ as the {\bf top cone} and the divisible part $\Z/p \{\Sigma^{-1} \frac{1}{u_\xi^{j} a_\xi^{k}} \} \oplus \Z/p\{\Sigma^{-1}\frac{\kappa_\xi}{u_\xi^{j}a_\xi^k}\}$ as the {\bf bottom cone}. 
\vspace{1cm}

We deduce the $p=2$ calculation in an analogus manner. There is only one non-trivial irreducible $C_2$-representation namely the sign representation $\sigma$. From Definition \ref{cohgen} we have two classes $u_\sigma \in  \tilde{H}^{\sigma-1}_G(S^0; \uZZ)$ and  $a_\sigma \in \tilde{H}^\sigma_{G}(S^0; \uZZ)$. Consider the Tate square \eqref{tate2}
\begin{myeq}\label{tate2}
\xymatrix{EG_+ \wedge H \uZZ\ar[d]^{\cong} \ar[r]& H \uZZ \ar[r]\ar[d] & \widetilde{EG}\wedge H \uZZ \ar[d] \\ EG_+ \wedge F(EG_+, H\uZZ) \ar[r] & F(EG_+, H\uZZ) \ar[r]& \widetilde{EG}\wedge F(EG_+, H\uZZ)}
\end{myeq}
As in the case of $p$ odd, we have isomorphisms 
$$\tilde{H}^\sigma_{G}({EG}_+; \uZZ) \to \tilde{H}^\sigma_{G}(S(2\sigma)_+; \uZZ)$$ 
and 
$$\tilde{H}^{\sigma}_G(S^0; \uZZ) \cong \tilde{H}^\sigma_{G}(S(2\sigma)_+; \uZZ).$$
This allows us to write using Proposition \ref{addstr} 
$$\pi_\bigstar F({EG}_+, H\uZZ) \cong \Z/2[a_\sigma, u_\sigma^\pm]$$
 as before. This yields 
$$\pi_\bigstar \widetilde{EG}\wedge F({EG}_+\wedge H\uZZ)\cong \Z/2[a_\sigma^\pm, u_\sigma^\pm]$$
 and the map
 $$\pi_\bigstar F({EG}_+, H\uZZ) \to \pi_\bigstar \widetilde{EG}\wedge F({EG}_+\wedge H\uZZ)$$
 is injective. Therefore, 
$$\pi_\bigstar {EG}_+\wedge H\uZZ \cong \pi_\bigstar {EG}_+\wedge F({EG}_+, H\uZZ) \cong \oplus \Z/2\{ \Sigma^{-1}\frac{u_\sigma^j}{a_\sigma^k}\} \text{ for } j \in \Z, k \geq 1.$$ 
Thus, in integer grading we obtain 
$$\pi_\ast {EG}_+\wedge H\uZZ \cong \oplus \Z/2\{ \Sigma^{-1}\frac{u_\sigma^k}{a_\sigma^k}\} \text{ for } k \geq 1.$$ 
This implies 
$$\pi_\ast \widetilde{EG} \wedge H\uZZ\cong \Z/2\{ 1\} \oplus \Z/2 \{\frac{u_\sigma^k}{a_\sigma^k} \}.$$
 Therefore, the $a_\sigma$-periodicity of $\pi_\bigstar \widetilde{EG} \wedge H \uZZ$ yields 
$$\pi_\bigstar \widetilde{EG} \wedge H \uZZ \cong Z/2[a_\sigma^\pm, u_\sigma].$$ 
Therefore, the kernel of the boundary 
$$\pi_\bigstar \widetilde{EG} \wedge H \uZZ \to \pi_{\bigstar-1} {EG}_+ \wedge H \uZZ$$
 is $\Z/2[a_\sigma, u_\sigma]$, and the cokernel is 
$$\oplus \Z/2\{\Sigma^{-1}u_\sigma^{-j}a_\sigma^{-k} \} \text{ for } j,k >0.$$ 
Therefore, 
\begin{prop}\label{cohpt2}
$$\pi_\bigstar (H \uZZ) \cong \Z/2 [a_\sigma, u_\sigma] \oplus \Z/2\{\Sigma^{-1}u_\sigma^{-j}a_\sigma^{-k} \} \text{ for } j,k >0.$$
\end{prop}
This method has been carried out for $H\uZ$ in \cite[Proposition 6.5]{MZ17} and \cite{JG17}.

\begin{prop} \label{bock}
The action of the Bockstein $\beta$ on $\tilde{H}^\bigstar_G(S^0)$ is given by the formula 
$$ \beta(\kappa_\xi) = a_\xi,~ \beta(u_\xi)=0,~ \beta(u_\sigma)=a_\sigma,~ \beta(a_\xi)=0,	~\beta(a_\sigma)=0.$$
\end{prop}

\begin{proof}
Observe from the computations in Propositions \ref{cohpt2} and \ref{cohptodd} that the map $q^\ast : \tHbs(S^0) \to \tHbs(EG_+)$ is injective on the top cone, so that we may prove the relations in the statement either in the cohomology of $S^0$ or in the cohomology of $EG_+$. We know that the Bockstein homomorphism is a derivation. Therefore, we have 
$$\beta(\kappa_\xi) = \beta(u_\xi x) = \beta(u_\xi)x + u_\xi \beta(x).$$
 So in order to prove the first equality, it is enough to check that $\beta(u_\xi) =0.$ Note that for the cofibre sequence 
 
 $$H\uZp \to H\uZpp \stackrel{\pi}{\to} H\uZp$$ 
 the Bockstein homomorphism fits into the cohomology long exact sequence 
 
$$\cdots \tilde{H}^{\xi-2}_G(S^0; \uZpp) \stackrel{\pi_\ast}{\to} \tilde{H}^{\xi-2}_G(S^0;\uZp) \stackrel{\beta}{\to} \tilde{H}^{\xi-1}_G(S^0; \uZp)\cdots$$
Since, the class $u_\xi \in \tilde{H}^\bigstar_G(S^0; \uZpp)$ can be represented by the map 
$$S^2 \to S^\xi\wedge H\uZpp$$ 
and also the class $u_\xi \in \tilde{H}^\bigstar_G(S^0; \uZp)$ represented analogously by 
$$S^2 \to S^\xi \wedge H\uZp$$ 
and they fit together into the following diagram 
$$\xymatrix{S^2 \ar[r]^-{u_\xi}\ar[dr]_-{u_\xi} & S^\xi \wedge H\uZpp \ar[d]^{1\wedge \pi} \\ 
                                                                    & S^\xi \wedge H\uZp}$$ 
Hence, we have $\pi_\ast(u_\xi) =u_\xi.$ This yields the Bockstein sends the class $u_\xi$ to zero.  Therefore, we obtain $\beta(\kappa_\xi) = a_\xi.$ This completes the proof for $p$ odd. 

In the case $p=2$, we again use the fact that $\beta$ is a derivation and $\beta^2=0$, so that it suffices to verify that $\beta(u_\sigma)=a_\sigma$. We directly compute a cell structure on $S^\sigma$ as $S^0 \cup C_2/e \times \DD^1$ with the boundary map generated by identity which may be identified as $C_2/e \times S^0 \to pt \times S^0$. Therefore, the reduced Bredon homology with constant coefficients $\Z/k$ is computed by the two term cell complex 
$$ \Z/k \stackrel{2}{\to} \Z/k$$
as the differential is generated by the transfer map which is multiplication by $2$. Directly comparing the answers for $k=2$ and $k=4$, we see that the map $H_1(S^\sigma; \underline{\Z/4}) \to H_1(S^\sigma; \uZZ)$ is $0$. Hence it follows that $u_\sigma$ does not lie in the image of $H^{\sigma-1}_G(S^0;\underline{\Z/4}) \to H^{\sigma -1}_G (S^0;\uZZ)$, and thus the only possibility is $\beta(u_\sigma)= a_\sigma$. 
\end{proof}

\subsection{Module structure of $\tilde{H}^{\bigstar}_{G}(X_+;\uZp)$}\label{module}
In this section, we compute the $\tilde{H}^\bigstar_G(S^0;\uZp)$-module structure of $\tilde{H}^\bigstar_G(X_+;\uZp)$. Note that this is determined by the action of the elements $u_\sigma, a_\sigma$ for $p=2$, and $u_\xi$, $\kappa_\xi$ and $a_\xi$ for $p$ odd. We write $q:X_+ \to S^0$ for the map which quotients out $X$ to a single point. Then, the elements in  $\tilde{H}^\bigstar_G(S^0;\uZp)$ act via multiplication by their images under $q^\ast$ which is a ring map. Therefore it suffices to compute the images of the generators under $q^\ast$. We have already done this for $X=EG$. Hence, we readily obtain 
\begin{prop}\label{mod}
1) Let $p$ be an odd prime. Then the action of 
$a_\xi, u_\xi \in \tilde{H}^\bigstar_{G}(S^0;\uZp)$ on $\tilde{H}^\bigstar_{G} ({EG}_+;\uZp) \cong \Z/p[x,a_\xi,u_\xi^\pm]/(x^2)$ is given by multiplication by the corresponding elements. The action of $\kappa_\xi$ equals multiplication by $xu_\xi$.   \\
2) Let $p=2$. Then the action of 
$a_\sigma, u_\sigma \in \tilde{H}^\bigstar_{G}(S^0;\uZZ)$ on $\tilde{H}^\bigstar_{G} ({EG}_+;\uZZ) \cong \Z/2[a_\sigma,u_\sigma^\pm]$ is given by multiplication by the corresponding elements. 
\end{prop} 

Next we consider a general free $G$-space $X$, which is non-equivariantly connected. We know that if $X$ had some point with a non-trivial stabilizer, then $S^0$ becomes an equivariant retract of $X_+$, so the map $q^\ast$ identifies $\tilde{H}^{\bigstar}_{G}(X_+;\uZp)$ as a subring which is included in $\tilde{H}^{\bigstar}_{G}(X^G_+;\uZp)$. So we restrict our attention to the free case, where $X \to X/G$ is a covering space with $G$ acting on $X$ by Deck transformations. This induces a homomorphism $\tau : \pi_1(X/G) \to G \cong \Z/p$, which is well-defined up to a unit in $\Z/p$. This gives an element of $\tilde{H}^1(X/G; \Z/p)$ which we also denote by $\tau$. We now have the following Proposition regarding the action of $\tilde{H}^\bigstar_G(S^0;\uZp)$ on $\tilde{H}^\bigstar_G(X_+;\uZp)$. 
\begin{prop}\label{modX}
1) Let $p$ be an odd prime. Then the action of 
$u_\xi \in \tilde{H}^\bigstar_{G}(S^0;\uZp)$ on $\tilde{H}^\bigstar_{G} (X_+;\uZp) \cong H^\ast(X/G;\Z/p) \otimes \Z/p[u_\xi^\pm]$ is given by multiplication by the corresponding element. The action of $\kappa_\xi$ equals multiplication by $\tau u_\xi$. The action of $a_\xi$ is given by multiplication by $\beta(\tau) u_\xi$.  \\
2) Let $p=2$. Then the action of 
$u_\sigma \in \tilde{H}^\bigstar_{G}(S^0;\uZZ)$ on $\tilde{H}^\bigstar_{G} (X_+;\uZZ) \cong H^\ast(X/G;\Z/2)\otimes \Z/2[u_\sigma^\pm]$ is given by multiplication by the corresponding element. The action of $a_\sigma$ is given by multiplication by $\tau u_\sigma$.  
\end{prop} 

\begin{proof}
We already know the action of $u_\xi$ (and $u_\sigma$ in the case $p=2$) from the calculation in Proposition \ref{addstr}. For proving the rest, we determine the images of $\kappa_\xi$ and $a_\xi$. If $p>2$, the action of $a_\xi$ follows from the action of $\kappa_\xi$ using the Bockstein as in Proposition \ref{bock}. Now, consider the projection map $q:X_+ \to S^0$ and its homotopy cofibre $C(q).$ As $X$ is (non-equivariantly) connected, it readily follows that $C(q)$ is non-equivariantly simply connected. 

For 1), we claim that the induced map 
$$q^\ast:\tilde{H}^{\xi-1}_G(S^0; \uZp) \to \tilde{H}^{\xi-1}_G(X_+; \uZp)$$ 
is injective. To prove the claim, it is enough to show that the group 
$$[C(q), S^{\xi-1}\wedge H\uZp]^{G}=0.$$ 
Note that this group fits into the long exact sequence 
\begin{myeq}\label{cp}
\xymatrix{\cdots [C(q),  \Sigma^{-1}H\uZp]^{G}\ar@{=}[d] \ar[r] & [C(q),  S^{\xi-1}\wedge H\uZp]^{G} \ar[r] & [C(q), S(\xi)_+\wedge H\uZp]^{G}\cdots\\ \tilde{H}^{-1}(C(q)/G; \Z/p)}
\end{myeq} 
associated to the cofibre sequence 
$$S(\xi)_+ \to S^0 \to S^\xi$$
Since, the group in the left is zero, therefore, we only remain to prove the group  $[C(q), S(\xi)_+\wedge H\uZp]^{G}$ vanishes. For this consider the cofibre sequence 
$$C_p/e_+ \to C_p/e_+ \to S(\xi)_+$$
It gives the long exact sequence
$$\xymatrix{\cdots [C(q), C_p/e_+ \wedge H\uZp]^G \ar@{=}[d]\ar[r] &  [C(q), S(\xi)_+\wedge H\uZp]^{G} \ar[r]  &[C(q), \Sigma C_p/e_+\wedge H\uZp]^{G} \ar@{=}[d]\cdots \\
 \tilde{H}^0(C(q); \Z/p) & & \tilde{H}^1(C(q); \Z/p)}$$
Since, $C(q)$ is simply connected, therefore, $[C(q), S(\xi)_+\wedge H\uZp]^{G}=0.$ Hence, using  \eqref{cp}, we obtain $[C(q),  S^{\xi-1}\wedge H\uZp]^{G}=0.$ Thus, we establish the claim. It follows that $q^\ast(\kappa_\xi)= \phi u_\xi$ for some $\phi \in H^1(X/G) = Hom(\pi_1(X/G),\Z/p)$ (using Proposition \ref{addstr}) which is non-zero. 

Now the above is true for any free $G$-space. Observe that the $1$-skeleton of $X$ is a union of copies of $S^1$ on which $G$ acts by multiplication by $p^{th}$ roots of $1$, which is equivariantly homeomorphic to $S(\xi)$. We, therefore, must have that $\phi$ pulls back non-trivially to the orbit space $S(\xi)/G$. This is also homeomorphic to $S^1$, with 
$$H^1(S(\xi)/G;\Z/p) \cong Hom(\pi_1(S(\xi)/G), \Z/p).$$
Hence, the pullback of $\phi$ must send $1\in \pi_1(S(\xi)/G)\cong \Z$ to a unit in $\Z/p$, and thus the kernel is the image of $\pi_1(S(\xi))$. As this is true for every map from $S(\xi)$ to $X$, the kernel of $\phi$ must be equal to the image of $\pi_1(X)$. Therefore, $\phi$ equals $\tau$ up to a unit of $\Z/p$. We may fix the choice of unit in the definition of $\tau$ so that $\phi=\tau$. Since the choice of $\phi$ is natural, this choice of $\tau$ is natural among equivariant maps between free $G$-spaces.      

For 2), consider the cofibre sequence ${C_2/e}_+ \to S^0 \to S^\sigma$ and its associated long exact sequence 
$$\xymatrix{ [C(q), H\uZZ ]^G \ar[r] \ar@{=}[d] & [C(q), S^\sigma \wedge H\uZZ]^G  \ar[r]  &[C(q), \Sigma {C_2}/e_+ \wedge H\uZZ]^G \ar@{=}[d] \cdots \\ 
\tilde{H}^0(C(q)/G, \Z/2) & & \tilde{H}^1(C(q), \Z/2)}$$

Since $C(p)$ is simply connected, the group $[C(q), S^\sigma \wedge H\uZZ]^G$ is trivial. Thus, the map 
$$\tilde{H}^\sigma_G(S^0; \uZZ) \to \tilde{H}^\sigma_G(X_+; \uZZ)$$
 is a monomorphism. Therefore, $q^\ast(a_\sigma)= \phi u_\sigma$ for some $\phi \in H^1(X/G;\Z/2)$.  Now we may apply a similar argument as above to deduce the result.
%
%
\end{proof}

\section{Applications to finite dimensional  $C_p$-spaces} 
We now restrict our attention to finite dimensional $C_p$-spaces, and focus on two kinds of examples -- the free case, and the case where the space is formed by attaching cells along representation spheres. In the former case, we obtain an invariant used to prove theorems about non-existence of equivariant maps. In the latter case, we obtain a freeness theorem that the  $RO(G)$-graded cohomology is a free module over the cohomology of point.  

\subsection{Cohomology of spheres} \label{tver}
Let $V$ be a fixed point free $C_p$-representation, that is, $V^G=0$. One has the cofibre sequence 
$$S(V)_+ \to S^0 \to S^V,$$
where the map $S^0 \to S^V$ induces the stable map $a_V$. Therefore, we have the associated long exact sequence
\begin{myeq}\label{exact}
\cdots \uH^{\alpha-1}_{C_p}(S(V)_+; \uZp)\to \uH^{\alpha-V}_{C_p}(S^0; \uZp) \stackrel{a_V.}{\to} \uH^{\alpha}_{C_p}(S^0; \uZp) \to \uH^{\alpha}_{C_p}(S(V)_+; \uZp)\cdots
\end{myeq} 
This allows us to compute  $\uH^\bigstar_{C_{p}}(S(V)_+; \uZp)$ from the formula for the ring structure on $\tilde{H}^{\bigstar}_{G}(S^0;\uZp)$ in Propositions \ref{cohptodd} and \ref{cohpt2}. We have the following useful Lemma. 
\begin{lemma}\label{sph}
Let $V$ be a fixed point free representation of $C_p$. Then, 
$$\uH^{V}_{G}(S(V)_+; \uZp)=0.$$
\end{lemma}
\begin{proof}
For $p=2$,  $V=\dim(V) \sigma$, and we have $a_V=a_\sigma^{\dim(V)}$. Proposition \ref{cohpt2} now implies that $\tH_G^0(S^0) \to \tH^V_G(S^0)$ induced by multiplication by $a_V$ is an isomorphism. For $p$ odd, $a_V= ua_\xi^{(\dim(V)/2)}$, where $u$ is a unit in $\pi_\bs H\uZp$ arising from the equivalence $S^{V-(\dim(V)/2)\xi}\wedge H\uZp \simeq H\uZp$.  It follows from Proposition \ref{cohptodd} that  $\tH_{C_p}^0(S^0) \to \tH^V_{C_p}(S^0)$ induced by multiplication by $a_V$ is an isomorphism. Now the result is implied directly by \eqref{exact}. 
%
%
\end{proof}

Lemma \ref{sph} allows us to conclude well-known Borsuk-Ulam type theorems \cite{IM00} for $G$. 
\begin{cor}\label{borsul}
Let $V$ and $V'$ be two fixed point free representations of $G$. Then there does not exist $G$-maps from $S(V) \to S(V')$ if $\dim(V) > \dim(V')$. 
\end{cor} 
\begin{proof}
We note that $\tilde{H}^{V'-V}_{G}(S^0;\uZp) = 0$, from the calculations of Propositions \ref{cohptodd} and \ref{cohpt2}, and the identification in \eqref{repfr}. If there was a $G$-map $S(V) \to S(V')$, then we obtain an induced map 
$$\tilde{H}^{\bigstar}_{G}(S(V')_+;\uZp) \to \tilde{H}^{\bigstar}_{G}(S(V)_+;\uZp)$$
of $\tilde{H}^{\bigstar}_{G}(S^0;\uZp)$-modules which sends $1$ to $1$. In degree $V'$, the left hand side is $0$ by Lemma \ref{sph}, and the right hand side is generated by $1\cdot a_{V'}$. This is a contradiction. 
\end{proof} 

We may also use these techniques to deduce a proof of the ``topological Tverberg conjecture" in the prime case. This states that for integers $n \geq 2$, $d \geq 1$ and $N=(d+1)(n-1)$, and for any continuous map $f : \Delta^{N} \to \R^d$, there exists $n$-pairwise disjoint faces $\sigma_1, \cdots, \sigma_n$ of the simplex $\Delta^{N}$ such that $f(\sigma_1) \cap \cdots\cap f(\sigma_n) \neq \emptyset$. This conjecture was first posed by B\'ar\'any, Shlosman, and Sz\"{u}cs \cite{B-S-S},  who proved if  $n$ is prime. Later, {\"O}zaydin \cite{Oz87}, Sarkaria \cite{sar}, and Volvikov \cite{vol96} using different techniques extended this result to $n$ a power of some prime. In \cite{fr15}, Frick describes a counterexample  when $n$ is not a prime power and $d \geq 3n+1$.  We have in earlier work observed \cite{BG17} that the methods in the prime power case do not help in proving weaker versions (that is, with increased values of $N$) if $n$ is not a prime power.

A map $f: \Delta^{(d+1)(n-1)} \to \mathbb{R}^d$ violating the topological Tverberg conjecture gives a $\Sigma_n$-equivariant map from the $((d+1)(n-1)+1)$-fold join $\{1, 2, \cdots, n \}^{\ast ((d+1)(n-1)+1)}$ to the representation sphere $S(W^{\oplus d})$ \cite{Oz87}, where $W$ is the standard representation of the symmetric group $\Sigma_n,$ of dimension $(n-1).$ For $n=p$, restricting to the cyclic subgroup $C_p$, we get a $C_p$-equivariant map from $\{1, \cdots, p \}^{\ast((d+1)(n-1)+1)}$ to $S(\bar{\rho}^{\oplus (d+1)})$, where $\bar{\rho}$ is the reduced regular representation.

Observe that the inclusion ${EC_p}^{((d+1)(p-1))} \subseteq EC_p$ induces an isomorphism in $\tilde{H}^\ast_{C_p}(-, \uZp)$ for $\ast \leq (d+1)(p-1)-1$ and is injective for $\ast = (d+1)(p-1).$ Since this is an inclusion of free $C_p$-spaces the result also holds for $\tilde{H}^\alpha_{C_p}(-; \uZp)$ where $\alpha \in RO(C_p)$ with $\dim \alpha \leq (d+1)(p-1).$ In particular,  observe for $1 \in \tilde{H}^0_{C_p}(EC_p; \uZp),$ Proposition \ref{mod} yields 
 $$a_{\frac{(d+1)(p-1)}{2} \xi} . 1= u_\xi^{\frac{(p-1)(d+1)}{2}} y^{\frac{(p-1)(d+1)}{2}}\neq 0,$$
thus, 
\begin{myeq} \label{acteg} 
a_{\frac{(d+1)(p-1)}{2} \xi} . 1\neq 0 \in \tilde{H}^{\frac{(d+1)(p-1)}{2}\xi}_{C_p}{(EC_p^{((p-1)(d+1))}}_+; \uZp).
\end{myeq}
Using these formulas, the following theorem provides a key step towards the topological Tverberg conjecture in the prime case.
\begin{thm}\label{obs}
There does not exist any $C_p$-map $EC_p^{((p-1)(d+1))} \to S(\bar{\rho}^{\oplus (d+1)}).$
\end{thm}

\begin{proof}
Suppose on contrary there is a $C_p$-map $f: EC_p^{((p-1)(d+1))} \to S(\bar{\rho}^{\oplus (d+1)}).$ Then it induces a $\tilde{H}^\bs_{C_p}(S^0)$-module map 
$$f^*: \tilde{H}^\bs_{C_p}(S(\bar{\rho}^{\oplus (d+1)})_+; \uZp) \to \tilde{H}^\bs_{C_p}({EC_p^{((p-1)(d+1))}}_+; \uZp)$$ 
So, using module structure, we have $$f^*(a_{\frac{(p-1)(d+1)}{2}\xi}.1)=a_{\frac{(d+1)(p-1)}{2}\xi}f^*(1).$$  Lemma \ref{sph} implies that $a_{\frac{(p-1)(d+1)}{2} \xi}.1=0.$ Therefore, this contradicts \eqref{acteg}.
\end{proof}

Finally, note that $EC_p^{(p-1)(d+1)}$ is a free $C_p$-space of dimension $(p-1)(d+1)$ and $\{1, \cdots, p \}^{\ast (p-1)(d+1)+1}$ has connectivity $(p-1)(d+1).$ Therefore, by $C_p$-equivariant obstruction theory we have a $C_p$-map from $EC_p^{(p-1)(d+1)} \to \{1, \cdots, p \}^{\ast (p-1)(d+1)+1}.$ Thus a $C_p$-map from $\{1, \cdots, p \}^{\ast (p-1)(d+1)+1} \to S(\bar{\rho}^{\oplus d+1})$ induces a $C_p$-map $EC_p^{((p-1)(d+1))} \to S(\bar{\rho}^{\oplus (d+1)}),$ contradicting Theorem \ref{obs}. This contradiction implies the topological Tverberg conjecture in the odd prime case. In fact, an analogous argument may be written also in the case $p=2$, but then the calculation is equivalent to Corollary \ref{borsul}.

 \subsection{Invariants of finite free $C_p$-spaces}\label{findim}
Let $X$ be a finite dimensional free $C_p$-space. Proposition \ref{addstr} implies that $\tilde{H}^\alpha_G(X_+; \uZp)=0$ for $\dim\alpha$ sufficiently large. Since the classes $\kappa_\xi$ and $a_\xi$ raise the total degree (for an element of $RO(G)$ we refer to the total degree as the dimension) of $\tilde{H}^\bigstar_G(X_+; \uZp)$ by one or two respectively, so, there is minimum degree $n(X)$ of $\kappa_\xi^\epsilon a_\xi^j$, $\epsilon \in \{ 0,1\}$ which acts trivially. That is, 
$$n(X) = \begin{cases}\min \{ 2j+\epsilon \mid \kappa_\xi^\epsilon a_\xi^j \text{ acts trivially on }\tilde{H}^\bigstar_G(X_+; \uZp) \}, & \text{ for $p$ odd} \\ 
\min \{ j \mid a_\sigma^j \text{ acts trivially on }\tilde{H}^\bigstar_G(X_+; \uZZ) \}, & \text{ for $p=2$.}
 \end{cases}$$
The number $n(X)$ behaves like an index of a $G$-space in the sense that if there is a $G$-map $X \to Y$, $n(Y) \geq n(X)$. For otherwise, the map $\tilde{H}^\bigstar_G(Y_+;\uZp) \to  \tilde{H}^\bigstar_G(X_+;\uZp)$ would not be a map of $\tilde{H}^\bigstar_G (S^0;\uZp)$-modules. 
  
We now relate $n(X)$ to the Fadell-Husenni index (\cite{FH}) for a $G$-space $X$. For a $G$-space $X$ and ring $R$, the Fadell-Husenni index of $X$ is defined to be the kernel ideal of the map $p: EG \times_G X \to BG$ in cohomology induced by the $G$-equivariant projection $X \to pt:$ 
 $$Index_G(X; R) = Ker(p^\ast: H^\ast(BG; R) \to H^\ast(EG \times_G X; R)).$$ 
If $G=C_p$ and $R=\Z/p$, note that the cohomology of $BG$ is free of rank $1$ in each degree. We write $i(X)$ to be the first integer $i$  where the degree $i$ part of $Index_G(X;R)$ is non-trivial. 
 \begin{thm}\label{F-H}
$n(X)=i(X)$. 
\end{thm}
\begin{proof}
For a free, finite dimensional $G$-space $X$ we have up to $G$-homotopy a unique map $\phi : X\to EG$. Also the Borel construction $X\times_G EG \simeq X/G$. Now consider the commutative triangle
 $$\xymatrix{\tilde{H}^\bigstar_G(EG_+, \uZp) \ar[rr]^{\phi^\ast} & & \tilde{H}^\bigstar_G(X_+, \uZp) \\ 
& \tilde{H}^\bigstar_G(S^0; \uZp)\ar[ur] \ar[ul]}$$ 
Since, both the space $X$ and $EG$ are free $G$-spaces, therefore, using Proposition \ref{addstr} the horizontal map turns out to 
$$p^\ast \otimes Id:H^\ast(BG)\otimes \bigotimes_{\xi \in \hat{G} \setminus \{1 \}} \Z/p [u_\xi^\pm] \to H^\ast(X/G)\otimes \bigotimes_{\xi \in \hat{G} \setminus \{1 \}} \Z/p [u_\xi^\pm].$$
Therefore, the result follows immediately from the Propostion \ref{mod}, and the fact that the elements $u_\xi$ have total degree $0$.
\end{proof}

\subsection{Freeness theorem}\label{free}
We use the calculations from Section \ref{coh} to prove a freeness result for certain $G$-cell complexes formed out of representations. One defines such a complex as a $Rep(G)$-complex as below. 
\begin{defn}
A $Rep(G)$-cell complex $X$ is a $G$-space with a filtration $\{ X^{(n)}\}_{n \geq 0}$ of subspaces such that \\
a) $X^{(0)}$ is a finite union of disjoint copies of $G/G$. \\
b) For each $n,$ $X^{(n+1)}$ is built up form $X^{(n)}$ by attaching cells of the form $D(V)$ along the boundary $S(V)$ with $\dim(V)=n+1$.  \\
c) $X = \cup_{n \geq 0} X^{(n)}$ has the colimit topology.
\end{defn}

Not all finite $G$-spaces are representable as $Rep(G)$-complexes where cells are attached sequentially in increasing dimension. Note that for $G$-spheres one may have equivariant maps (not null-homotopic) which increase the total dimension. For example one has the map 
$$a_\xi : S^V \to S^{V\oplus \xi}$$
which is non-trivial in cohomology if $V$ is fixed point free (Proposition \ref{addstr}). In fact the mapping cone of $a_\xi : S^0 \to S^\xi$ is $\Sigma S(\xi)_+$, which is easily verified to have non-free cohomology. In fact $\Sigma 1\in \tilde{H}^1_G(\Sigma S(\xi)_+;\uZp)$ is not divisible by $a_\xi$ or $u_\xi$ but satisfies $a_\xi \cdot \Sigma 1 =0$. One easily verifies analogously that the cohomologies $\tilde{H}^\bigstar_G(S(k\xi)_+;\uZp)$ are all non-free.

There are plenty of examples of $Rep(G)$-complexes. The one point compactification $S^V$ written as union of $D(V)$ identifying $S(V)$ to a point is one. Observe from \cite[Section 8.1]{BG19} that $\C P(V)$, $Gr_k(V)$ are also $Rep(G)$-complexes for a unitary $G$-representation $V$.

Our freeness result holds only for $Rep(G)$-complexes such that there are no cells in consecutive dimensions. In this context we make the following definition. 
\begin{defn}
A $Rep(G)$ complex is said to be sparse if it does not have cells in consecutive dimensions. 
\end{defn}

An example of a sparse $Rep(G)$-complex is a complex with only even dimensional cells, but one may observe easily that these are not all. The freeness result we prove here generalizes the analogous result for $p=2$ in \cite{Kr10}. 

\begin{rmk}\label{Clover-countex} 
The sparseness condition is not present in the case $p=2$, however, it is necessary for a freeness theorem in the $p$ odd case. This is due to the following example of Clover May of a $Rep(C_p)$-complex with non-free cohomology. We have  a $C_p$-map $\lambda : S^\xi \to S^2=S^\xi/C_p$. Let $Y$ be the mapping cone of $\lambda$, so that it is a representation complex. The map $\lambda$ induces an isomorphism on $H^2_{C_p}$, so that $\tilde{H}^n_{C_p}(Y)=0$ for all $n$. Now one may observe that in $\tilde{H}^\bigstar_{C_p}(S^0)$, for every $k$ there is a $n$ such that  $\tilde{H}^{k\xi -n}_{C_p}(S^0)\neq 0$. It follows that the cohomology of $Y$ cannot have a summand of the type $\tilde{H}^{\bigstar - V}_{C_p}(S^0)$ and hence, cannot be free. 
\end{rmk}

We now build up the techniques leading to a proof of the freeness result. In view of the equivalence $H\uZp \wedge S^{\xi^i} \simeq H\uZp \wedge S^{\xi^j}$ for $p\nmid i, j$, we may assume that all the representations $V$ in the cell complex decomposition contain only the trivial representations, and the representation $\xi$. 

We begin the proof with a localization theorem, which identifies the cohomology of finite $G$-spaces after we invert the element $a_\xi$. This is the $p$ odd version of \cite[Lemma 4.3]{CM18} with a very similar proof. In the rest of the section $p$ is an odd prime, and as a notation we write $\tilde{H}^\bigstar_G(Y)$ for $\tilde{H}^\bigstar_G(Y;\uZp)$.
\begin{prop}\label{loc}
For a finite $C_p$-space $X,$ we have the following equivalence given by the $a_\xi$-localization:
$$a_\xi^{-1}\tilde{H}^\bigstar_{C_p}(X_+) \cong a_\xi^{-1} \tilde{H}^\bigstar_{C_p}(X^{C_p}_+)\cong H^\ast(X^{C_p}; \Z/p)\otimes a_\xi^{-1}H^\bigstar_{C_p}(S^0).$$
\end{prop}

\begin{proof}
Consider the inclusion $X^{C_p} \to X$, which induces the map 
$$a_\xi^{-1} \tilde{H}^\bigstar_{C_p}(X_+) \to a_\xi^{-1}\tilde{H}^\bigstar_{C_p}(X^{C_p}_+),$$
between cohomology theories on locally finite, finite $C_p$-CW complexes. The left hand side is a cohomology theory because the localization functor is exact. The fixed points functor preserves cofibre sequences and so, $a_\xi^{-1}\tilde{H}^\bigstar_{C_p}((-)^{C_p}_+)$ is a cohomology theory. They agree on $G$-orbits, so the first isomorphism follows. The second isomorphism is clear as $X^G$ is a trivial $G$-space.
\end{proof}


Let $X$ be a $Rep(G)$-complex not necessarily sparse. Also, assume that $\tilde{H}^\bigstar_G(X^{(n-1)}_+)$ is a free $\tilde{H}^\bigstar_G(S^0)$-module with generators $\omega_1, \cdots, \omega_s$ in the $RO(G)$-degree $W_1, \cdots, W_s$ respectively and $X^{(n)}$ is formed by attaching exactly one cell $D(V)$ to $X^{(n-1)}$. We have the following long exact sequence 
$$\cdots \tilde{H}^\bigstar_G(S^V) \to \tilde{H}^\bigstar_G(X^{(n)}_+) \to \tilde{H}^\bigstar_G(X^{(n-1)}_+) \stackrel{d}{\to} \tilde{H}^{\bigstar+1}_G(S^V) \cdots $$
We denote by $\nu$ the generator in $\tH^V_G(S^V)$. The following lemma gives a method to prove freeness in the case there are cells in $\dim(V)-1$ provided some fixed point criteria are satisfied. 
\begin{prop}\label{topcone}
Assume that  $\dim(W_1)= \dim(V)-1$ and that $\dim W_1^{C_p}< \dim V^{C_p}$. After a change of basis among $\omega_i$ of dimension $\dim(V)-1$, either one of the following holds \\
1) $d(\omega_1)=0$.\\
2) $\omega_1 = \nu -1 $, $d(\omega_1)=\nu$ and $d(\omega_i)=0$ for all $i\geq 2$.  
\end{prop}

\begin{proof} 
In order to prove the Lemma, we assume $d(\omega_1)\neq 0$ and then prove the conclusion in 2). If $\dim(W_1) = \dim(V) - 1$, we must have $W_1 + 1 - V$ is of the form $k(\xi -2)$ for $k>0$, so that $d(\omega_1)=u_\xi^k \nu$ up to units. 

After rearranging $\omega_i$ if necessary we assume that $\omega_1$ is the one with the least value of $-|(W_i +1-V)^{C_p}|$ (this equals $2k$ where $W_i +1 -V = k(\xi -2)$)  among $\{\omega_i \mid \dim(\omega_i)=\dim(\nu)-1,~ \dim\omega_1^{C_p}< \dim \nu^{C_p}, ~ d(\omega_i)\neq 0\}.$ Since $\omega_1$ has the least value of $k$ we have $d(u_\xi^t\omega_1) = d(\omega_i)$ for some $t$, as $\omega_i$ varies over the set above. We now change $\omega_i$ to $\omega_i - u_\xi^t \omega_1$ to assume $d(\omega_i)=0$. 

For the other $\omega_j$ we have $d(\omega_j)=0$ or 
$$d(\omega_j)= \Sigma^{-1}\frac{1}{a_\xi^j u_\xi^l} \nu = d( \Sigma^{-1}\frac{1}{a_\xi^j u_\xi^{l+k}} \omega_1)$$
or,  
$$d(\omega_j)= \Sigma^{-1}\frac{\kappa_\xi}{a_\xi^j u_\xi^l} \nu = d( \Sigma^{-1}\frac{\kappa_\xi}{a_\xi^j u_\xi^{l+k}} \omega_1)$$
In each case we may add a multiple of $\omega_1$ to $\omega_j$ to ensure $d(\omega_j)=0$. Therefore, $\omega_1$ is the only class with a non-trivial differential. We now are reduced to the simpler long exact sequence  
\begin{myeq}\label{cellext}
\cdots \tilde{H}^\bigstar_G(S^V) \stackrel{q^\ast}{\to} \tilde{H}^\bigstar_G(X_+) \stackrel{i^\ast}{\to} \tilde{H}^\bigstar_G(S^W) \stackrel{d}{\to} \tilde{H}^{\bigstar+1}_G(S^V) \cdots.
\end{myeq}
putting $\omega=\omega_1$ and suppressing the other $\omega_i$ with trivial differential. We consider \eqref{cellext} for $\bigstar = V$, assume that the conclusion of the lemma does not hold, and show that the boundary map $d : \tilde{H}^{V-1}_G(S^W) \to \tilde{H}^V_G(S^V)$ is trivial. The group $\tilde{H}^{V-1}_G(S^W)$ is generated by $\Sigma^{-1}\frac{\kappa_\xi}{a_\xi u_\xi^t}\omega$ for some $t$. If $d$ is non-trivial on this class, we must have  
$$d(\Sigma^{-1}\frac{\kappa_\xi}{a_\xi u_\xi^t}\omega) = \nu$$ 
up to an unit in $\Z/p.$ Note that this class is again divisible by $a_\xi$, so $d(\Sigma^{-1}\frac{\kappa_\xi}{a_\xi^2 u_\xi^t}\omega) \in \tilde{H}^{V -\xi}_G(S^V)$, which is zero. Now the module structure implies 
$$0= a_\xi d(\Sigma^{-1}\frac{\kappa_\xi}{a_\xi^2 u_\xi^t} \omega) = d(\Sigma^{-1}\frac{\kappa_\xi}{a_\xi u_\xi^t} \omega) = \nu$$ 
which gives a contradiction. Hence, $q^\ast(\nu)\neq 0$. For degree reasons, also $q^\ast(a_\xi^s \nu)\neq 0$. Therefore, the class $q^\ast(\nu)$ would survive to the $a_\xi$-localization and have $u_\xi$ torsion, which cannot happen in the cohomology of a finite space by Lemma \ref{loc}.
\end{proof}

 We start at the crucial step in the proof for a cohomological two cell complex: This is defined to be a $G$-space $X$ for which  there is a  cofibre sequence
$$\Sigma^{-V}H\uZp \to F(X_+,H\uZp) \to \Sigma^{-W} H\uZp$$
with $\dim(W)<\dim(V)-1$. This gives the long exact sequence
\begin{myeq}\label{2cell}
\cdots \tilde{H}^\bigstar_G(S^V) \to \tilde{H}^\bigstar_G(X_+) \to \tilde{H}^\bigstar_G(S^W) \stackrel{d}{\to} \tilde{H}^{\bigstar+1}_G(S^W) \cdots
\end{myeq}
As before, we write the generators of $\tilde{H}^W_G(S^W)$ and $\tilde{H}^V_G(S^V)$ as $\omega$ and $\nu$ respectively. Since $d$ is an $\tilde{H}^\bigstar_G(S^0)$-module map, the computation is determined by $d(\omega) \in \tilde{H}^{W+1}_G(S^V).$

\begin{lemma}\label{twocell}
Let $X$ be a cohomological $2$-cell complex as above such that $\dim(W)<\dim(V) -1$. Then, $\tilde{H}^\bigstar_G(X_+)$ is a free $\tilde{H}^\bigstar_G(S^0)$-module. In particular, one of the following must holds \\
1. $\tilde{H}^\bigstar_G(X_+) \cong \Sigma^W \tilde{H}^\bigstar_G(S^0)\oplus \Sigma^V \tilde{H}^\bigstar_G(S^0)$.\\ 
2. $\tilde{H}^\bigstar_G(X_+)$ is free with two generators one with the same dimension as $\omega$ and the other with the same dimension as $\nu$.
\end{lemma}
\begin{proof}
If $d(\omega)=0$, the conclusion 1) holds. Otherwise from Proposition \ref{cohptodd} we have that the boundary should be either 
$$d(\omega) = \Sigma^{-1} \frac{1}{u_\xi^j a_\xi^k} \nu \text{ for } j , k  \geq 1$$
or
$$ d(\omega) = \Sigma^{-1} \frac{\kappa_\xi}{u_\xi^{j} a_\xi^k} \nu \text{ for } j , k  \geq 1.$$
Since the Bockstein of the class $\omega$ is zero, the Bockstein of right hand side of the boundary must be zero. This forces that boundary should be one of the following form\\
$i)$ $d(\omega) = \Sigma^{-1} \frac{1}{u_\xi^j a_\xi^k} \nu$ for $j , k  \geq 1.$\\
$ii)$ $d(\omega) = \Sigma^{-1} \frac{\kappa_\xi}{u_\xi^{j} a_\xi} \nu$ for $j \geq 1.$\\
In the latter case, $\dim(W)=\dim(V)-1$. Therefore, it suffices to consider the case $i)$. We must have the relation
$$W = V +j(2-\xi)-k\xi.$$
The long exact sequence \eqref{2cell} for $\bigstar = \alpha : = W+ j(\xi -2)$ gives the isomorphism 
$$\xymatrix{\tilde{H}^\alpha_G(X_+)\ar[r]^(.30){i^\ast}_(.30){\cong} & \tilde{H}^{W+ j(\xi-2)}_G(S^W) \cong \Z/p\{ u_\xi^j\omega\}}. $$ 
So, there exists a class $a \in \tilde{H}^\alpha_G(X_+)$ such that $i^\ast(a) = u_\xi^j \omega.$ Also, for $\bigstar= \beta = V + j(2-\xi)= W +k \xi$,  the same exact sequence \eqref{2cell} gives 
$$\xymatrix@C=0.7cm{\tilde{H}^{W-1+k\xi}_G(S^W)\ar@{=}[d]\ar[r]^-{d} 
&  \tilde{H}^{V+j(2-\xi)}_G(S^V) \ar[r]^-{q^\ast}  \ar@{=}[d]
&\tilde{H}^\alpha_G(X_+) \ar[r]^-{i^\ast} 
& \tilde{H}^{W+k\xi}_G(S^W) \ar[r] 
& \tilde{H}^{V+1+j(2-\xi)}_G(S^V) \ar@{=}[d]
\\ \Z/p\{a_\xi^{k-1}\kappa_\xi \omega \}\ar[r]^-{d} 
& \Z/p\{ \Sigma^{-1} \frac{\kappa_\xi}{a_\xi u_ \xi^j}\nu\}
& & & 0}$$

Since $d(\omega) = \Sigma^{-1} \frac{1}{u_\xi^{j} a_\xi^k}\nu$, the boundary map $d: \tilde{H}^{W-1+k \xi}_G(S^W) \to \tilde{H}^{V+j(2-\xi)}_G(S^V)$ is an isomorphism. This implies a class $b \in \tilde{H}^{\beta}_G(X_+)$ such that $i^\ast(b) = a_\xi^k \omega$. Observe that $\dim(\alpha)=\dim(W)$ and $\dim(\beta)=\dim(V)$. 

\addtocounter{equation}{1}
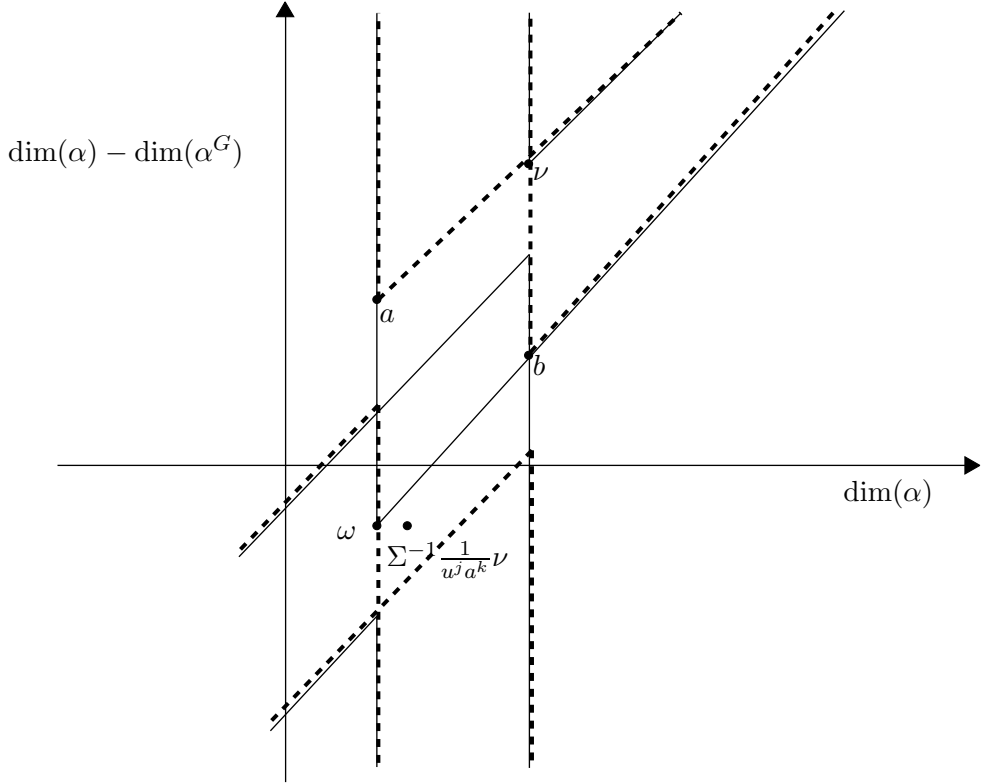
\begin{figure}[htp]
 \centering
\begin{tikzpicture}[x=2.0cm,y=2.0cm]
\draw [line width=0.5pt] (1,4)-- (7,4);
\draw [line width=0.5pt] (2.5,7.)-- (2.5,1.9);
\draw (6.10048947953825,3.975432814993677) node[anchor=north west] {$\dim(\alpha)$};
\draw (0.608769489964604,6.2542209566744935) node[anchor=north west] {$\dim(\alpha) -\dim(\alpha^G)$};
\draw (2.771608474617447,3.6657040027159544) node[anchor=north west] {$\omega $};
\draw (3.095173093862263,3.5657040027159544) node[anchor=north west] {$\Sigma^{-1}\frac{1}{u^ja^k}\nu$};
\draw [line width=0.5pt] (3.1,3.6)-- (6.174479165230286,7.014767157031732);
\draw [line width=0.5pt] (3.1,3.6)-- (3.100171626163344,6.997775870820408);
\draw [line width=0.5pt] (4.1,6.)-- (5.1,7.);
\draw (4.0573475303056075,6.0457015353833885) node[anchor=north west] {$\nu$};
\draw [line width=0.5pt] (4.103129214296232,6.003129214296233)-- (4.102789658778592,7.001403621127376);
\draw [line width=0.5pt] (4.103129214296232,5.398639859816693)-- (4.103129214296232,1.9941667839310293);
\draw [line width=0.5pt] (4.103129214296232,5.398639859816693)-- (2.1931987186232828,3.3939982927591877);
\draw [line width=0.5pt] (3.1,3)-- (3.1,2.);
\draw [line width=0.5pt] (3.1,3)-- (2.4,2.24);
\draw (4.0573475303056075,4.8) node[anchor=north west] {$b$};
\draw [dashed][line width = 1.5pt]  (4.11,4.8) -- (4.11,7);
\draw [dashed][line width = 1.5pt]  (4.11,4.75) -- (6.1,7);
\draw [dashed][line width = 1.5pt]  (4.12,4.1) -- (4.12,2);
\draw [dashed][line width = 1.5pt]  (4.12,4.1) -- (2.4,2.3);
\draw [dashed][line width = 1.5pt]  (3.11,5.1) -- (3.11,7);
\draw [dashed][line width = 1.5pt]  (3.11,5.1) -- (5.1,7);
\draw [dashed][line width = 1.5pt]  (3.11,4.4) -- (3.11,2);
\draw [dashed][line width = 1.5pt]  (3.11,4.4) -- (2.2,3.44);
\draw (3.05, 5.1) node[anchor= north west]{$a$};
\begin{scriptsize}
\draw [fill=black] (4.097840285680781,5.99975812878811) circle (1.5pt);
\draw [fill=black] (3.1,3.6) circle (1.5pt);
\draw [fill=black] (3.3,3.6) circle (1.5pt);
\draw [fill=black] (4.097840285680781,4.73) circle (1.5pt);
\draw [fill=black] (3.1,5.1) circle (1.5pt);
\draw [fill=black,shift={(7,4.)},rotate=270] (0,0) ++(0 pt,3.75pt) -- ++(3.2475952641916446pt,-5.625pt)--++(-6.495190528383289pt,0 pt) -- ++(3.2475952641916446pt,5.625pt);
\draw [fill=black,shift={(2.496751500528123,7.007762781275576)}] (0,0) ++(0 pt,3.75pt) -- ++(3.2475952641916446pt,-5.625pt)--++(-6.495190528383289pt,0 pt) -- ++(3.2475952641916446pt,5.625pt);

\end{scriptsize}
\end{tikzpicture}
\label{Fig1}
\caption{The classes $\omega$, $\nu$, $a$ and $b$ together with the top cones and bottom cones on them. The continuous lines represent the cones on $\omega$ and $\nu$ while the dashed ones represent the cones on $a$ and $b$.}
\end{figure}

For each $i,$ the class $a_\xi^i b$ is non zero, and hence, it'll survive $a_\xi$-localization. So, by Lemma \ref{loc} the classes $u_\xi^i b$ are non trivial for each $i.$ From the exact sequence \eqref{2cell} at $\bigstar = V$ we get the following extension
$$0 \to \Z/p\{ \nu\} \stackrel{q^\ast}{\to} \tilde{H}^V_G(X_+) \stackrel{i^\ast}{\to} \Z/p\{ a_\xi^k u_\xi^j \omega\} \to 0$$
Observe that both $a_\xi^k a$ and $u_\xi^jb$ map to the same element, so that up to units, 
\begin{myeq}\label{equ1}
q^\ast(\nu)= a_\xi^k a - u_\xi^{j}b.
\end{myeq}
Therefore, the classes $a_\xi^k a$ and $u_\xi^{j}b$ generate $\tilde{H}^V_G(X_+).$ The map $q^\ast$ is an $\tilde{H}^\bigstar_G(S^0)$-module homomorphism, so that $a$ and $b$ generate the image of the top cone of $\nu$ in $X$. For the bottom cone of $\nu$ note that (using \eqref{equ1})
$$q^\ast(\Sigma^{-1}\frac{\kappa_\xi}{a_\xi^{k+1}u_\xi}\nu) = \Sigma
^{-1}\frac{\kappa_\xi}{a_\xi u_\xi}a$$
and
$$q^\ast(\Sigma^{-1}\frac{\kappa_\xi}{a_\xi u_\xi^{j+1}}\nu) = -\Sigma
^{-1}\frac{\kappa_\xi}{a_\xi u_\xi}b,$$
so that these are also generated by the classes $a_\xi^k a$ and $u_\xi^j b.$ In particular, we can construct a $\tilde{H}^\bigstar_G(S^0)$-module map $f$ from a free $\tilde{H}^\bigstar_G(S^0)$-module with two generators $\eta$ and $\zeta$ in degrees $V +k \xi$ and $V + j(2-\xi)$ respectively to $\tilde{H}^V_G(X_+)$ with $f(\eta) = a$ and $f(\zeta) =b.$ This $f$ is surjective by the above, and by comparing degree-wise ranks we see that this is an isomorphism. Therefore, we obtain the conclusion 2).
\end{proof}

The next technique involves extending Lemma \ref{twocell} to the case with many cells provided the $d(\omega_i)$ does not involve the class $\kappa_\xi$. This argument is an adaptation of \cite[Theorem 3.15]{Fer99}.  

\begin{prop}\label{inn_ind}
Let $M^\bs$ be a graded $\tHbs(S^0)$-module such that it fits in the long exact sequence
$$\xymatrix{\cdots\tHbs(S^0)\{\nu\} \ar[r] & M^\bs \ar[r]^-{i^\ast} & \bigoplus_{i=1}^s\tHbs(S^0)\{\omega_i\} \ar[r]^-d & \tHbsb(S^0)\{\nu\} \cdots}$$
with for some $1 \leq n \leq s,$ the classes $\omega_1, \cdots, \omega_s, \nu$ satisfy the following\\
(i) $\dim W_1<\cdots \dim W_n< \dim V.$\\
(ii) $\dim V^{C_p}<\dim{W_1}^{C_p}<\cdots <\dim{W_n}^{C_p}.$ \\
(iii) $d(\omega_i)=0$ for $i >n.$\\
(iv) $d(\omega_i)= \Sigma^{-1} \frac{1}{u_\xi^{m_i} a_{\xi}^{k_i}}\nu$ for $1 \leq i \leq n.$\\
Then there exist $\omega'_1, \cdots, \omega'_s, \nu'$ such that\\
(a) If $l\leq n$, $i^\ast(\omega'_l)= u_\xi^{n_l}\omega_l + \sum_{j=1}^{l-1} Q_j(u_\xi, a_\xi)\omega_j$ where $n_l\geq 0$ and $Q_j(u_\xi, a_\xi)$ is a polynomial  in $u_\xi, a_\xi.$ If $l> n$, $i^*({\omega_l}')= \omega_l$. \\
(b) $i^\ast(\nu')= \sum_{i=1}^s P_i(u_\xi, a_\xi)\omega_i$, where $P_i(u_\xi, a_\xi)$ is a polynomial  in $u_\xi, a_\xi.$\\
(c) $M^\bs \cong \tHbs(S^{V'})\oplus \bigoplus_{i=1}^s\tHbs(S^{W'_i}).$\\
\end{prop}

\begin{proof}
We use induction on $s$. For $s=1$ using the computation of $2$-cell case we get  $W'_1= W_1+ m_1(\xi-2)$ and $V'_1 = V + m_1(2-\xi)$ such that $M^\bs \cong \tHbs(S^0)\{\nu'_1\}\oplus \tHbs(S^0)\{\omega'_1\}.$ In the general case we can run the $2$-cell case with $\omega_1$ and $\nu$ to obtain the diagram

\begin{myeq}\label{twocellind}
\xymatrix{\cdots \tHbs(S^V)\ar@{=}[d] \ar[r]^-{(q_1^\ast,q_2^\ast)} &\tHbs(S^{V_1'})\oplus \tHbs(S^{W_1'})\ar[d]^{(\theta_1^\bs, \theta_2^\bs)} \ar[r]^-{i^\ast=(i_1^\ast, i_2^\ast)} & \tHbs(S^{W_1})\ar[d]^{incl} \ar[r]^{d_1} &\tHbsb(S^V)\cdots \ar@{=}[d]\\ \cdots \tHbs(S^V) \ar[r]^-{q^\ast} &M^\bs \ar[r]^-{i^\ast} & \tHbs(S^{W_1})\oplus  \bigoplus_{i=2}^n \tHbs(S^{W_i}) \ar[r]^-{d=(d_1, d_2)} & \tHbsb(S^V)\cdots}
\end{myeq} 

It gives the long exact sequence 
\begin{myeq}\label{coker-ex}
\xymatrix{\cdots\tHbs(S^{V'_1}) \ar[r] & coker(\theta_2^\bs) \ar[r]^-{i_c^\ast} & \bigoplus_{i=2}^s\tHbs(S^{W_i}) \ar[r]^-{\Sigma q_1^\ast\circ d_2} & \tHbsb(S^V) \cdots}
\end{myeq}

Note that $q_1^\ast(\nu) = u_\xi^{m_1} \nu'_1$ therefore, $\Sigma q_1^\ast\circ d_2(\omega_i) = \Sigma^{-1} \frac{1}{u_\xi^{m_i-m_1}a_\xi^{k_i}} \nu'_1$ for $2 \leq i \leq n.$ So, $coker(\theta_2^\bs)$ satisfies the induction hypothesis.  Therefore, there exist classes $\omega'_2, \cdots, \omega'_s, \nu'$ such that $i_c^\ast(\omega'_i)= u_\xi^{n_i}\omega_i + \sum_{j=2}^{i-1}Q_j(u_\xi, a_\xi)$ for some non-negative integer $n_i$ and  $i_c^\ast(\nu')= \sum_{i=2}^s P_i(u_\xi, a_\xi)\omega_i.$ Moreover,  $coker(\theta_2^\bs) \cong \tHbs(S^{V'})\oplus \bigoplus_{i=2}^s\tHbs(S^{W'_i}).$

Now using \eqref{twocellind} and \eqref{coker-ex} we have the following commutative digram
\begin{myeq}\label{i-diam}
\xymatrix{0 \ar[r] & \tHbs(S^{W'_1})\ar[d]^{i_2^\ast} \ar[r]^{\theta_2^\bs} & M^\bs \ar[r]\ar[d]^{i^\ast} & coker(\theta_2^\bs)\ar[d]^{i_c^\ast} \ar[r] & 0\\
& \tHbs(S^{W_1}) \ar[r]^-{incl} & \bigoplus_{i=1}^s \tHbs(S^{W_i}) \ar[r]^{pr_2} & \bigoplus_{i=2}^s \tHbs(S^{W_i}) }
\end{myeq}
 such that the rows are exact. Since $coker(\theta_2^\bs)$ is free by induction hypothesis, therefore, the exactness of the top row yields $M^\bs \cong \tHbs(S^{V'})\oplus \bigoplus_{i=1}^s\tHbs(S^{W'_i}).$ 
 
 By the left commutative square of \eqref{i-diam} it is immediate that $i^\ast(\omega'_1)= u_\xi^{m_1}\omega_1.$ Now the right commutative square yields 
$$pr_2\circ i^\ast(\omega'_i) = u_\xi^{n_i}\omega_i+ \sum_{j=2}^{i-1}Q_j(u_\xi, a_\xi)\omega_j = pr_2(u_\xi^{n_i}\omega_i+ \sum_{j=2}^{i-1}Q_j(u_\xi, a_\xi)\omega_j).$$
This implies
$$i^\ast(\omega'_i)-(u_\xi^{n_i}\omega_i+ \sum_{j=2}^{i-1}Q_j(u_\xi, a_\xi)\omega_j) \in ker(pr_2)\cong \tHbs(S^{W_1}).$$
 Therefore, for some $x_i \in \tilde{H}_G^{W'_i-W_1}(S^0),$ we have $$i^\ast(\omega'_i)= x_i.\omega_1+u_\xi^{n_i}\omega_i+ \sum_{j=2}^{i-1}Q_j(u_\xi, a_\xi)\omega_j.$$
For $2 \leq i \leq n,$ $\dim(W'_i - W_1) = \dim(W_i -W_1) = 2(k_1-k_i)>0.$ Therefore, by dimension reasons $x_i$ can't have a term involving $\kappa_\xi$ and the divisible part of $\tHbs(S^0)$. Hence, $x_i$ is a polynomial in $u_\xi,a_\xi.$ For $i>n,$ we have $i^\ast(\omega'_i) =\omega_i.$

 Similarly, $pr_2\circ i^\ast(\nu') = \sum_{i=2}^s P_i(u_\xi, a_\xi)\omega_i = pr_2(\sum_{i=2}^s P_i(u_\xi, a_\xi)\omega_i).$ Therefore, for some $y \in \tilde{H}_G^{V'-W_1}(S^0)$ we have
 $$i^\ast(\nu') = y. \omega_1 + \sum_{i=2}^s P_i(u_\xi, a_\xi)\omega_i.$$ 
 Again note that $\dim(V'-W_1)= \dim(V-W_1) = 2k_1.$ Therefore, by dimension reason $y$ is a polynomial in $u_\xi, a_\xi.$
 This completes the induction.
\end{proof}

Finally we use Proposition \ref{inn_ind} to prove the freeness result for sparse complexes. 
\begin{thm}\label{frcp}
Suppose $X$ is a sparse  $Rep(G)$-cell complex, then $\tilde{H}_G^\bigstar (X_+; \uZp)$ is a free $\tilde{H}^\bigstar_G(S^0; \uZp)$-module whose basis elements lie in $\ker(\beta)$.   
\end{thm}

\begin{proof}
We use induction on the cellular filtration $\{X^{(n)}\}_{n\geq 0}$ of $X$. The base case is the $0$-skeleton which is a wedge of copies of $S^0$. Since $X$ is locally finite, it suffices to prove the case of a single cell attachment: $X$ is obtained from $Y$ by attaching a single cell $D(V)$ for some representation $V$, where $Y$ is a $Rep(G)$-complex such that 
$$\tilde{H}^\bigstar_G(Y_+) \cong \oplus_{i=1}^s \tilde{H}^\bigstar_G(S^{W_i}) \cong \oplus_{i=1}^s \tilde{H}^{\bigstar-W_i}_G(S^0),$$
with $\dim(W_i)<\dim(V)-1$ and $\beta(\omega_i)=0$. We use the long exact sequence 
\begin{myeq}\label{cellatt}
\cdots \tilde{H}^\bigstar_G(S^V) \to \tilde{H}^\bigstar_G(X_+) \to \tilde{H}^\bigstar_G(Y_+) \stackrel{d}{\to} \tilde{H}^{\bigstar+1}_G(S^V) \cdots
\end{myeq}
The cohomology of $X$ is determined by $d(\omega_i)\in \tilde{H}^{W_i+1}_G(S^V)\cong \tilde{H}^{W_i+1-V}_G(S^0)$. 
%
%
%
%
 
The boundary map $d$ is a $\tilde{H}^\bigstar_G(S^0; \uZp)$-module map. We first observe that after rearrangement there is a maximal $n$ and elements $\omega_1, \cdots, \omega_n$ satisfying the following properties\\
(i) $d\omega_i \neq 0$ for $1\leq i \leq n.$ \\
(ii) $\dim W_1 < \cdots <\dim W_n < \dim V$\\
(iii) $\dim V^{C_p} < \dim W_1^{C_p}< \cdots < \dim W_n^{C_p}.$ \\ 
(iv) After a base change, $d\omega =0$ for the free module generators other than $\{\omega_1, \cdots, \omega_n\}$.\\
 If $\omega$ is not one of $\omega_1, \cdots, \omega_n$; then there exists $1\leq j\leq n$ such that $\dim W \geq \dim W_j $ and $\dim W^G \leq \dim W_j^{G}.$ In particular, 
$$d(\omega_j) =\lambda \Sigma^{-1} \frac{1}{u_\xi^{m_j }a_\xi^{k_j}}\nu $$
 with $k \leq k_j$,  $m \leq m_j$ and $\lambda \in \Z/p$. Therefore, we can choose a change of basis as 
$$\{\omega- u_\xi^{m_j-m}a_\xi^{k_j-k}\omega_j\mid \omega \in   \{\omega_{n+1}, \cdots, \omega_s\}\} \cup  \{\omega_1, \cdots, \omega_n\}$$ 
so that $d(\omega) =0$ for $\omega \notin \{\omega_1, \cdots, \omega_n\}.$ For the $\omega_i $ as above the boundary is given by 
$$d(\omega_i)= \Sigma^{-1} \frac{1}{u_\xi^m a_\xi^k}\nu$$ 
for some $m, k \geq 1$ (by the Bockstein argument as in Lemma \ref{twocell}). Now we apply Proposition \ref{inn_ind} to deduce that $\tHbs(X_+)$ is a free module on $\omega_i'$ and $\nu'$. In order to complete the induction we need to verify that $\beta(\omega_i')=0$. This is where the sparseness property comes in. 

Observe from Proposition \ref{inn_ind} that $i^\ast(\beta(\nu') )=0$ so that $\beta(\nu')= x.q^\ast\nu$. Since $\dim V'=\dim V$, $\dim V'^{C_p} > \dim V^{C_p}$, and that $H^\alpha_G(S^0)=0$ for $\dim \alpha=1$ and $\dim\alpha^{C_p}>0$, we obtain $\beta(\nu')=0$. Analogously $i^\ast(\beta (\omega_l'))=0$, which implies $\beta(\omega_l')=xq^\ast(\nu)$. We note that for degree reasons $x$ must lie in the divisible part so that 
$$x= \begin{cases} 
\Sigma^{-1}\frac{1}{u_\xi^da_\xi^e} ~~~ \mbox{ if } \dim(V)- \dim(W_l') \mbox{ is even.} \\ 
\Sigma^{-1}\frac{\kappa_\xi}{u_\xi^da_\xi^e} ~~~ \mbox{ if } \dim(V)- \dim(W_l') \mbox{ is odd.} \end{cases}$$
In the former case note that we have a class $\tau$ namely $\Sigma^{-1}\frac{\kappa_\xi}{u_\xi^da_\xi^{e+1}}\nu$ such that $\beta(q^\ast(\tau))=\beta(\omega_l')$. Then we may replace $\omega_l'$ by $\omega_l' - q^\ast (\tau)$ which changes the basis of the free module to a new basis (this follows as $\tau$ lies in the divisible part of $\nu$). 

In the latter case, note that 
$$\beta(\Sigma^{-1}\frac{\kappa_\xi}{u_\xi^da_\xi^e}) =\Sigma^{-1}\frac{1}{u_\xi^da_\xi^{e-1}}$$
so that the class $\zeta = \Sigma^{-1}\frac{1}{u_\xi^da_\xi^{e-1}}$ must lie in the image of $d$. Such a class cannot be hit by the image of $\omega_i$ if $\dim(W_i) > \dim(W_l')+1$. If this class is a multiple of $d(\omega_i)$ for $\dim(W_i)\leq \dim(W_l')$ then also $\Sigma^{-1}\frac{\kappa_\xi}{u_\xi^da_\xi^{e}}$ is a multiple. Therefore the only option is $\dim(W_i)=\dim(W_l')+1$, but that violates the sparseness condition. 
\end{proof}

\bibliographystyle{siam}
\bibliography{algtop}{}

\mbox{ }\\

\end{document}